\documentclass[11pt]{article}
\usepackage{color}
\usepackage{graphicx}
\usepackage{enumerate}
\usepackage{amsfonts,amsmath,latexsym,amssymb}
\usepackage{amsthm}
\newtheorem{theorem}{Theorem}[section]
\newtheorem{corollary}{Corollary}[section]
\newtheorem{lemma}{Lemma}[section]
\newtheorem{proposition}{Proposition}[section]
\newtheorem{remark}{Remark}[section]
\newtheorem{definition}{Definition}[section]

\numberwithin{equation}{section}

\begin{document}
\title{On the power means and Lawson-Lim means for positive invertible operators}
\author{Wenshi Liao$^a$\footnote{Corresponding author. E-mail: liaowenshi@gmail.com, jlwu678@163.com, 527628144@qq.com.}, Junliang Wu$^a$ and Haisong Cao$^a$\\
$^a${\small College of Mathematics and Statistics, Chongqing University, Chongqing 401331, P.R. China}}

\date{}
\maketitle
{\bf Abstract.} This note aims to
present some reverse inequalities about the power means and Karcher mean via the Kantorovich constant and some of these have been generalized to higher power. Also, we generalize the reverse weighted arithmetic-geometric mean inequality of $n$ positive invertible operators due to Lawson and Lim. In addition, we make comparisons between the Karcher mean and Lawson-Lim geometric mean for higher power.
\vspace{3mm}

\vspace{3mm}


\vspace{3mm}

\section{Introduction}
Let $B(\mathcal{H})$ be the $C^*$-algebra of all bounded linear operators on a complex separable Hilbert space $\mathcal{H}$. $B(\mathcal{H})^+$ stands for the set of positive elements in $B(\mathcal{H})$.  A linear map $\Phi$: $B(\mathcal{H})\rightarrow B(\mathcal{K})$ is said to be positive if $\Phi(A)\ge 0$ whenever $A\ge0$.
A positive linear map is said to be normalized (unital) if it maps $I_n$, the identity operator, to $I_m$.
Note that a positive linear map $\Phi$ is monotone in the sense that $A\le B$ implies $\Phi(A)\le \Phi(B)$.
 $\mathbb{P}$ stands for the convex cone of positive invertible operators. $\Delta_n$ denotes the simplex of positive probability vectors in $\mathbb{R}^n$ convexly spanned by the unit coordinate vectors. $\|\cdot\|$ and $|||\cdot|||$ denote the operator norm and the unitarily invariant norm, respectively. $\mathrm{tr}$ is the trace functional.

In \cite{Lim}, Lim and P\'{a}lfia have proposed power means of positive definite matrices
 and their notion and most of their results readily extended to the setting of positive invertible operators on a complex Hilbert space (See \cite{Lawson2}):
 \begin{definition}{\rm(\textbf{Power means})}
Let $\mathbb{A}=(A_1,A_2,\cdots,A_n)\in \mathbb{P}^n$ and $\omega\in \Delta_n$.
For $t\in[-1,1]$,
the power means $P_t(\omega;\mathbb{A})$ is defined as the unique
positive definite solution of the following non-linear equations:
\[
X=\sum\limits^n_{i=1}\omega_i(X\sharp_tA_i), ~for~ t\in(0,1],~~~~~
\]
\[
~~~~~~X=\sum\limits^n_{i=1}\omega_i(X^{-1}\sharp_{-t}A_i^{-1})^{-1}, ~for~ t\in[-1,0),
\]
where $A\sharp_t B :=A^{\frac{1}{2}}(A^{-\frac{1}{2}}BA^{-\frac{1}{2}})^{t} A^{\frac{1}{2}}$
is the $t$-weighted geometric mean of $A$ and $B$. $P_t(\omega;\mathbb{A})$ is called
the $\omega$-weighted power mean of order $t$ of $A_1,A_2,\cdots,A_n$.  To simplify the notation,
we write $P_t(\mathbb{A})=P_t(\frac{1}{n},\frac{1}{n},\cdots,\frac{1}{n};\mathbb{A})$.
\end{definition}

Since the power means $P_t(\omega;\mathbb{A})$ is increasing for $t\in [-1,1]\setminus \{0\}$,
 it interpolates between the weighted arithmetic mean and harmonic mean: $P_1(\omega;\mathbb{A})=\sum_{i=1}^nw_iA_i$
and $P_{-1}(\omega;\mathbb{A})=(\sum_{i=1}^nw_iA_i^{-1})^{-1}$.

Let $t\in(0,1]$ and $f:\mathbb{P}\rightarrow \mathbb{P}$ defined by $f(X)=\sum^n_{i=1}\omega_i(X\sharp_tA_i)$.
Then by the L\"{o}ewner-Heinz inequality, $f$ is monotone: $X\le Y$ implies $f(X)\le f(Y)$.
By Theorem 3.1 and Remark 3.4 of \cite{Lawson3}, $f$ is a strict contraction for the Thompson metric and
has a unique fixed point by the Banach fixed point theorem:
\[\lim_{k\rightarrow\infty}f^k(X)=P_t(\omega;\mathbb{A}), ~X\in\mathbb{P}.\]

Since the pioneering papers of Pusz and Woronowicz \cite{Pusz}, Ando \cite{Ando3}, and Kubo and Ando \cite{Kubo},
an extensive theory of two-variable geometric mean has sprung up for positive operators:
For two positive operators $A$ and $B$, the operator geometric mean is defined by
$A\sharp B :=A^{\frac{1}{2}}(A^{-\frac{1}{2}}BA^{-\frac{1}{2}})^{\frac{1}{2}} A^{\frac{1}{2}}$
for $A>0$.  Once one realizes that the matrix geometric mean $A\sharp B$ is the metric midpoint of $A$ and $B$
for the trace metric on the set of positive definite matrices of some fixed dimension (see, e.g., \cite{Bhatia, Lawson2}).
The operator geometric mean has many characterizations. For example,
\[A\sharp B=\max\left\{X\mid X=X^*,
\left[
  \begin{array}{cc}
    A  & X  \\
    X & B \\
  \end{array}
\right]\ge0
\right\}
\]
and $A\sharp B$ is the unique positive solution of the Riccati equation $XA^{-1}X = B$. Moreover, it is monotone,
jointly concave and congruence invariant and $A\sharp B=B\sharp A$.
But the $n$-variable case for $n> 2$ was a long standing problem and many authors studied the geometric mean of $n$-variable.

In 2004, Ando et al.\cite{Ando1} succeeded in the formulate of the geometric mean for $n$ positive definite matrices,
and they showed that it satisfies ten important properties:
\begin{definition}
{\rm(\textbf{Ando-Li-Mahthias geometric mean}\cite{Ando1})}
Let $A_i$ $(i=1,2,\cdots,n)$ be positive definite matrices.
Then the geometric mean $G_{ALM}(A_1,A_2,\cdots,A_n)$ is defined by induction as follows:

$~\mathrm{(i)}$ $G_{ALM}(A_1,A_2)=A_1\#A_2$.

$\mathrm{(ii)}$ Assume that the geometric mean of any $n-1$-tuple of operators is defined. Let
\[
G_{ALM}((A_j)_{j\neq i})=G_{ALM}(A_1,\cdots,A_{i-1},A_{i+1},\cdots,A_n),
\]
\indent \indent\and let sequences $\{A^{(r)}_i\}^\infty_{r=0}$ be $A^{(0)}_i=A_i$ and $A^{(r)}_i =G_{ALM}((A^{(r-1)}_j)_{j\neq i})$. If there exists
\indent \indent $\lim_{r\rightarrow \infty}A^{(r)}_i$, and it does not depend on $i$, then the geometric mean of $n$-matrices is
\indent \indent defined as
\[\lim\limits_{r\rightarrow \infty}A^{(r)}_i=G_{ALM}(A_1,A_2,\cdots,A_n).\]
\end{definition}
In [15], Yamazaki pointed out that the definition of the geometric mean by Ando, Li and
Mathias can be extended to Hilbert space operators.
Lawson and Lim \cite{Lawson2} established a definition of the weighted version of
the Ando-Li-Mahthias geometric mean for $n$ positive operators, we call it \textbf{Lawson-Lim geometric mean}. Following \cite{Lawson2},
we recall the definition of higher order weighted geometric mean $G[n,t]$ with $t\in(0,1)$ for $n$
positive operators $A_1,A_2,\cdots,A_n$. Let
$G[2,t](A_1,A_2)=A_1\sharp_t A_2 =A_1^{\frac{1}{2}}(A_1^{-\frac{1}{2}}A_2A_1^{-\frac{1}{2}})^{t} A_1^{\frac{1}{2}}$
(the unique geodesic curve containing $A$ and $B$ and its unique metric midpoint
 $A\sharp B=A\sharp_{\frac{1}{2}}B$ is the geometric mean of $A$ and $B$).
For $n\ge 3$, $G[n,t]$ is defined inductively as follows: Put $A_i^{(0)}=A_i$ for all $i=1,2,\cdots,n$ and
\[
A_i^{(r)}=G[n-1,t]\left(\left(A_j^{(r-1)}\right)_{j\neq i}\right)=G[n-1,t]\left(A_1^{(r-1)},\cdots,A_{i-1}^{(r-1)},A_{i+1}^{(r-1)},\cdots,A_n^{(r-1)}\right)
\]
inductively for $r$. Then the sequences $\{A_i^{(r)}\}$ have the same limit for all $i=1,2,\cdots,n$ in the Thompson metric.
So $G[n,t](A_1,A_2,\cdots,A_n)=\lim_{r\rightarrow\infty}A_i^{(r)}$.
In particular, $G[n,\frac{1}{2}]$ for $t=\frac{1}{2}$ is the Ando-Li-Mathias geometric mean.

Similarly, the weighted arithmetic mean is defined as follows: Put $\widetilde{A_i^{(0)}}=A_i$ for all $i=1,2,\cdots,n$ and
\[
\widetilde{A_i^{(r)}}=A[n-1,t]\left(\left(\widetilde{A_j^{(r-1)}}\right)_{j\neq i}\right)=A[n-1,t]\left(\widetilde{A_1^{(r-1)}},\cdots,\widetilde{A_{i-1}^{(r-1)}},\widetilde{A_{i+1}^{(r-1)}},\cdots,\widetilde{A_n^{(r-1)}}\right)
\]
inductively for $r$. Then the sequences $\{\widetilde{A_i^{(r)}}\}$ have the same limit for all $i=1,2,\cdots,n$.
If we put $A[n,t](A_1,A_2,\cdots,A_n)=\lim_{r\rightarrow\infty}\widetilde{A_i^{(r)}}$, then it is expressed by
\[A[n,t](A_1,A_2,\cdots,A_n)=t[n]_1A_1+t[n]_2A_2+\cdots+t[n]_nA_n,\]
where $t[n]_i\ge0$ for all $i=1,2,\cdots,n$ with $\sum_{i=1}^n t[n]_i=1$. Also, the weighted harmonic mean
$H[n,t](A_1,A_2,\cdots,A_n)$ is defined as
\[H[n,t](A_1,A_2,\cdots,A_n)=\left(t[n]_1A_1^{-1}+t[n]_2A_2^{-1}+\cdots+t[n]_nA_n^{-1}\right)^{-1}.\]
Note that the coefficient $\{t[n]_i\}$ depends on $n$ and $t$ only, see \cite{Fujii, Seo} for more details.

Moreover, the weighted arithmetic-geometric-harmonic mean inequality holds:
\begin{equation}
\label{agh}
H[n,t](A_1,A_2,\cdots,A_n)\le G[n,t](A_1,A_2,\cdots,A_n) \le A[n,t](A_1,A_2,\cdots,A_n).
\end{equation}
Since then, another approach to generalizing the geometric mean to $n$-variables, depending on Riemannian trace metric, was the Karcher mean, which was studied by many researchers, see \cite{Lawson3, Lim} and the reference therein.
Let $\mathbb{A}=A_1, A_2,\cdots,A_n\in \mathbb{P}^n$ and $\omega=(w_1,w_2,\cdots, w_n)\in\Delta_n$. By computing appropriate derivatives as in \cite{Bhatia, Moakher},
 the $\omega$-weighted \textbf{Karcher mean} of $\mathbb{A}$, denoted by $G_K(\omega;\mathbb{A})$,
coincides with the unique positive definite solution of the Karcher equation
\begin{equation}
\label{Kar}
\sum_{i=1}^nw_i\log(X^{\frac{1}{2}}A_i^{-1}X^{-\frac{1}{2}})=0.
\end{equation}
In the two operators case, $A_1, A_2\in \mathbb{P}$, the Karcher mean coincides with the weighted geometric mean $A_1\sharp_t A_2 =A_1^{\frac{1}{2}}(A_1^{-\frac{1}{2}}A_2A_1^{-\frac{1}{2}})^{t} A_1^{\frac{1}{2}}$.
From \eqref{Kar}, the Karcher mean satisfies the self-duality
$G_K(\omega;\mathbb{A})=G_K(\omega;\mathbb{A}^{-1})^{-1},$
where $\mathbb{A}^{-1}=(A^{-1}_1,A^{-1}_2,\cdots,A^{-1}_n)$.
Lim and P\'{a}lfia \cite{Lim} also established that the Karcher mean is the limit of power means as $t\rightarrow0$ in the finite-dimensional setting:
\begin{equation}
\label{Lim}
\lim_{t\rightarrow0} P_t(\omega;\mathbb{A})=G_K(\omega;\mathbb{A}),
\end{equation}
then Lawson and Lim \cite{Lawson3} showed \eqref{Lim} is valid in the infinite-dimensional setting.

\section{Preliminary}

Tominaga \cite{Tominaga} and Ali\'{c} et al.\cite{Alic} showed the following non-commutative arithmetic-geometric mean inequality:
For positive invertible operators $A_1$ and $A_2$ such that $m\le A_1, A_2\le M$ for some scalars $0<m\le M$ and $h=\frac{M}{m}$, one has
\begin{equation}\label{tom}
(1-v)A_1+vA_2\le S(h)A_1\sharp_vA_2,
\end{equation}
where $S(h)=\frac{(h-1)h^{\frac{1}{h-1}}}{e\log h}(h\neq1)$ is the Specht ratio and $S(1)=1$.
By \eqref{tom}, we can easily obtain the following one for any positive unital linear map $\Phi$:
\begin{equation}\label{tom1}
\Phi(A_1)\sharp_v\Phi(A_2)\le (1-v)\Phi(A_1)+v\Phi(A_2)=\Phi((1-v)A_1+vA_2)\le S(h)\Phi(A_1\sharp_v A_2).
\end{equation}

Fujii et al.\cite{Fujii} showed the following reverses of the weighted
arithmetic-geometric mean inequality of $n$ positive invertible operators due to Lawson and Lim:
\begin{theorem}\label{th21}
For any integer $n\ge 2$, let $A_1,A_2,\cdots,A_n$ be positive invertible operators
such that $m\le A_i\le M$ for all $i=1,2,\cdots,n$ and some scalars $0<m\le M$.
Then for each $t\in(0,1)$
\begin{equation}\label{th211}
~~A[n,t](A_1,\cdots, A_n)\le \frac{(m+M)^{2}}{4mM}G[n,t](A_1,\cdots, A_n)
\end{equation}
and
\begin{equation}\label{th212}
A[n,t](A_1,\cdots, A_n)\le S^2(h)G[n,t](A_1,\cdots, A_n).~~~
\end{equation}
\end{theorem}
Fujii \cite{Fujii1} and Lin \cite{Lin1} state a relation between the Specht ratio and the Kantorovich constant severally:
For $0<m\le M$ and $h=\frac{M}{m}$,
\begin{equation}\label{SK}
S(h)\le \frac{(m+M)^{2}}{4mM}\le S^2(h).
\end{equation}

It is obvious that \eqref{th211} is tighter than \eqref{th212} via the inequality \eqref{SK}.

Seo \cite{Seo} presented the following inequalities for the $n$-variables in terms of the higher order
weighted geometric mean due to Lawson and Lim, which are an extension and a converse of Ando's inequality
 $\Phi(A\sharp_t B)\le \Phi(A)\sharp_t \Phi(B)$
(See \cite[Theorem 4.1.5]{Bhatia} and \cite{Ando2}) for a positive unital linear map $\Phi$ and $t\in[0,1]$, respectively:
\begin{theorem}\label{th22}
Let $\Phi$ be a positive unital linear map on $B(\mathcal{H})$ and
let $A_1, A_2, \cdots, A_n \in \mathbb{P}^n$ for any positive integer $n\ge 2$ on a Hilbert space $\mathcal{H}$ such that $m\le A_i\le M$ for $i=1,2,\cdots,n$
and some scalars $0<m<M$.
Then for each $t\in(0,1)$,
\[
\Phi(G[n,t](A_1,\cdots, A_n))\le G[n,t](\Phi(A_1),\cdots, \Phi(A_n))\le \frac{(m+M)^{2}}{4mM}\Phi(G[n,t](A_1,\cdots, A_n)).
\]
\end{theorem}
On the other hand, Bhatia \cite{Bhatia3} derive an inequality of the $\omega$-weighted Karcher mean of $n$ positive definite matrices for any positive unital linear map $\Phi$ that
\begin{equation}\label{Bhatia}
\Phi(G_K(\omega;A_1,\cdots, A_n))\le G_K(\omega;\Phi(A_1),\cdots, \Phi(A_n)).
\end{equation}
Also, Lim \cite{Lim} obtained the above inequality via the power means. We will be engaged to show the reverse inequality of \eqref{Bhatia}.

Although the Ando-Li-Mahthias geometric mean does not coincide with the Karcher mean in general, Fujii and Seo \cite{Fujii1} made a comparison between the Ando-Li-Mahthias geometric mean and the Karcher mean and also obtained an inequality for the Ando-Li-Mahthias geometric mean:
\begin{theorem} \cite{Fujii1}
For any integer $n\ge 2$, let $A_1,A_2,\cdots,A_n$ be positive definite matrices in $\mathbb{P}$
such that $m\le A_i\le M$ for all $i=1,2,\cdots,n$ and some scalars $0<m\le M$.
Then
\begin{equation}\label{th231}
\frac{4mM}{(m+M)^{2}}G_K(A_1,\cdots, A_n)\le G_{ALM}(A_1,\cdots, A_n)\le\frac{(m+M)^{2}}{4mM}G_K(A_1,\cdots, A_n)~~~
\end{equation}
and
\begin{equation}\label{th232}
G_{ALM}(A^p_1,\cdots,A^p_n)\le \frac{(m+M)^{2p}}{4^pm^pM^p}G_{ALM}(A_1,\cdots, A_n)^p~~for~all~0<p<1.
\end{equation}
\end{theorem}
The case of $p>1$ about \eqref{th232} will be one of motivations for our study.

\section{The power means and Karcher mean}
It is well known that the Power means and Karcher mean are relevant via \eqref{Lim}. In this section, we will obtain some generalized properties and inequalities with respect to the Power means and Karcher mean.

Firstly, we list some important properties of the power means, some of which appeared in \cite {Lawson3, Lim}:
\begin{proposition}\label{pro}
Let $\mathbb{A}=(A_1,\cdots, A_n)\in \mathbb{P}^n$, $\mathbb{A}^{-1}=(A^{-1}_1,A^{-1}_2,\cdots,A^{-1}_n) \in \mathbb{P}^n$, a weight vector $\omega=(w_1,\cdots,w_n)\in \Delta_n$
and $t\in[-1,1]\setminus \{0\}$.
Then the power means satisfies the following properties:

 $(\mathrm{P1})$ (Duality) $P_{-t}(\omega;\mathbb{A}^{-1})^{-1}=P_{t}(\omega;\mathbb{A});$

$(\mathrm{P2})$ (Homogeneity) $P_t(\omega;a\mathbb{A})=a P_t(\omega;\mathbb{A});$

 $(\mathrm{P3})$ (Continuous) $\lim_{t\rightarrow 0}P_t(\omega;\mathbb{A})=G_K (\omega;\mathbb{A});$

 $(\mathrm{P4})$ ( APH weighted mean inequalities)  $(\sum_{i=1}^nw_iA_i^{-1})^{-1}\le P_t(\omega;\mathbb{A})\le \sum_{i=1}^nw_iA_i;$

 $(\mathrm{P5})$ If $t\in(0,1]$, then $\Phi(P_t(\omega;\mathbb{A}))\le P_t(\omega;\Phi(\mathbb{A}))$ for any positive unital linear map $\Phi$,
where
\indent\ \indent\ $\Phi(\mathbb{A})=(\Phi(A_1),\cdots,\Phi(A_n))$.
If $t\in [-1,0)$, then $P_t(\omega;\Phi(\mathbb{A}))\le\Phi(P_t(\omega;\mathbb{A}))$ for any
\indent\ \indent\ strictly positive unital linear map $\Phi;$

$(\mathrm{P6})$ For $t\in(0,1]$,
\[
\mathrm{tr}\left(P_t(\omega;\mathbb{A})\right)\le \left(\sum\limits_{i=1}^nw_i(\mathrm{tr}A_i)^t\right)^{\frac{1}{t}} ~and~
\mathrm{tr}\left(P_{-t}(\omega;\mathbb{A})\right)\ge n\left(\sum\limits_{i=1}^nw_i(\mathrm{tr}A^{-1}_i)^t\right)^{-\frac{1}{t}}.
\]
\end{proposition}
\begin{proof}
We provide a proof of $(\mathrm{P6})$. The other properties are known in \cite{Lawson3, Lim}.

By the definition of the power means, let $X=P_t(\omega;\mathbb{A})$ for $t\in(0,1]$. Then
\[\mathrm{tr}X=\sum_{i=1}^nw_i\mathrm{tr}(X\sharp_tA_i)\le \sum_{i=1}^nw_i(\mathrm{tr}X)^{1-t}(\mathrm{tr}A_i)^t=(\mathrm{tr}X)^{1-t}\sum_{i=1}^nw_i(\mathrm{tr}A_i)^t,\]
where the inequality follows from the Corollary 9 of \cite{Sra}.
Thus, $\mathrm{tr}X\le \left(\sum_{i=1}^nw_i\mathrm{tr}(A_i)^t\right)^{\frac{1}{t}}$.

Since
$n=\mathrm{tr}I\le \mathrm{tr}A \cdot\mathrm{tr}A^{-1}$ for an $n$-dimensional operator $A> 0$ (See \cite[Theorem 6.2.2]{Wang}),
\[\mathrm{tr}(P_{-t}(\omega;\mathbb{A}))=\mathrm{tr}(P_{t}(\omega;\mathbb{A}^{-1})^{-1})\ge n \left(\mathrm{tr}P_{t}(\omega;\mathbb{A}^{-1})\right)^{-1}\ge n\left(\sum\limits_{i=1}^nw_i(\mathrm{tr}A^{-1}_i)^t\right)^{-\frac{1}{t}}.\qedhere\]
\end{proof}

\begin{remark}
Let $\Phi(A)=\frac{\mathrm{tr}A}{n}$ for an $n$-dimensional operator $A> 0$ in $(\mathrm{P5})$ when $t\in(0,1]$ or by the first inequality in Proposition \ref{pro}{\rm (P6)}.  As $t\rightarrow 0$, by \eqref{Lim}, then we have
\[
\mathrm{tr}(G_K(\omega;\mathbb{A}))\le \prod_{i=1}^n(\mathrm{tr}A_i)^{w_i}.
\]
\end{remark}
\begin{remark}
Let $\Phi(A)=\frac{\mathrm{tr}A}{n}$ in the Choi's inequality $\Phi(A)^{-1}\le \Phi(A^{-1})$ for an $n$-dimensional operator $A> 0$ (See \cite[Theorem 2.3.6]{Bhatia}). Then we have $\mathrm{tr}A^{-1}\ge n^2\left(\mathrm{tr}A\right)^{-1}.$
 So we can obtain a stronger result than the second inequality in {\rm (P6)} as follows:
 \[
 \mathrm{tr}(P_{-t}(\omega;\mathbb{A}))\ge n^2\left(\sum\limits_{i=1}^nw_i(\mathrm{tr}A^{-1}_i)^t\right)^{-\frac{1}{t}}.
\]
\end{remark}

The next Proposition generalize properties (P6).

\begin{proposition}\label{th32}
Let $A_1, A_2, \cdots, A_n\in \mathbb{P}^n$ be $n$-dimensional operators such that $m\le A_i\le M$ for $i=1,2,\cdots,n$
and some scalars $0<m<M$. Then for $t\in(0,1]$,
\begin{equation}\label{trace1}
\mathrm{tr}(P_t(\omega;\mathbb{A}))\ge \left(\frac{4mM}{(m+M)^2} \sum\limits_{i=1}^nw_i(\mathrm{tr}A_i)^t\right)^{\frac{1}{t}}~~~~~
\end{equation}
and
\[
~~~~~~~~~~~~\mathrm{tr}(P_{-t}(\omega;\mathbb{A}))\le n^2\cdot \left(\frac{(m+M)^2}{4mM}\right)^{1+\frac{1}{t}}\left(\sum\limits_{i=1}^nw_i(\mathrm{tr}A^{-1}_i)^t\right)^{-\frac{1}{t}}.
\]
\end{proposition}

\begin{proof}
By Theorem \ref{th22} with $n=2$, the following inequality
\[
\Phi(A)\sharp_t\Phi(B)\le \frac{(m+M)^2}{4mM}\Phi(A\sharp_t B)
\]
holds. Let $\Phi(A)=\frac{\mathrm{tr}A}{n}$ for $A>0$ in the above inequality and $X=P_t(\omega;\mathbb{A})$ for $t\in(0,1]$ in the definition of power means. Then
\begin{align*}
\mathrm{tr}X&=\sum_{i=1}^nw_i\mathrm{tr}(X\sharp_tA_i)\\&\ge \frac{4mM}{(m+M)^2}\sum_{i=1}^nw_i(\mathrm{tr}X)^{1-t}(\mathrm{tr}A_i)^t\\
&=\frac{4mM}{(m+M)^2}(\mathrm{tr}X)^{1-t}\sum_{i=1}^nw_i(\mathrm{tr}A_i)^t,
\end{align*}
therefore,
\[
(\mathrm{tr}X)^t\ge \frac{4mM}{(m+M)^2}\sum_{i=1}^nw_i(\mathrm{tr}A_i)^t.
\]
Taking $\Phi(A)=\frac{\mathrm{tr}A}{n}$ in the Kantorovich inequality $\Phi(A^{-1})\le\frac{(m+M)^2}{4mM}\Phi(A)^{-1}$ for any $A > 0$(See \cite[Proposition 2.7.8]{Bhatia}), we have
\[\mathrm{tr}A^{-1}\le \frac{n^2(m+M)^2}{4mM}(\mathrm{tr}A)^{-1}~ \mathrm{for}~ A> 0.\]
Hence,
\begin{align*}
\mathrm{tr}(P_{-t}(\omega;\mathbb{A}))&=\mathrm{tr}(P_{t}(\omega;\mathbb{A}^{-1})^{-1})\\
&\le \frac{n^2(m+M)^2}{4mM} (\mathrm{tr}(P_{t}(\omega;\mathbb{A}^{-1})))^{-1}\\
&\le n^2\left(\frac{(m+M)^2}{4mM}\right)^{1+\frac{1}{t}}\left(\sum\limits_{i=1}^nw_i(\mathrm{tr}A^{-1}_i)^t\right)^{-\frac{1}{t}}.\qedhere
\end{align*}
\end{proof}
\begin{remark}
By \eqref{tom1} and a similar way as in the proof of Proposition \ref{th32}, we derive
\[
\mathrm{tr}(P_t(\omega;\mathbb{A}))\ge \left(S(h)^{-1} \sum\limits_{i=1}^nw_i(\mathrm{tr}A_i)^t\right)^{\frac{1}{t}}.
\]
From \eqref{SK}, we know that the above inequality is sharper than \eqref{trace1}.
\end{remark}

\begin{theorem}\label{th31}
Let $\Phi$ be a positive unital linear map on $B(\mathcal{H})$ and
let $A_1, A_2, \cdots, A_n \in \mathbb{P}^n$ for any positive integer $n\ge 2$ on a Hilbert space $\mathcal{H}$ such that $m\le A_i\le M$ for $i=1,2,\cdots,n$
and some scalars $0<m<M$. Then for $t\in(0,1]$,
\begin{equation}
\label{th311}
\sum\limits_{i=1}^nw_i\Phi(A_i)\le \frac{(m+M)^2}{4mM}\Phi(P_t(\omega;\mathbb{A}))~~~~~~
\end{equation}
and
\[
~~~~~~~P_t(\omega;\Phi(\mathbb{A}))\le\frac{(m+M)^2}{4mM}\left(\sum\limits_{i=1}^nw_i\Phi(A_i)^{-1}\right)^{-1}.
\]
\end{theorem}
\begin{proof}
By the Kantorovich inequality (See \cite[Proposition 2.7.8]{Bhatia}) and the APH weighted mean inequalities $(\mathrm{P4})$, we can easily obtained that
\[
\sum\limits_{i=1}^nw_i\Phi(A_i)=\Phi\left(\sum\limits_{i=1}^nw_iA_i\right)\le \frac{(m+M)^2}{4mM} \Phi\left(\left(\sum\limits_{i=1}^nw_iA_i^{-1}\right)^{-1}\right)\le\frac{(m+M)^2}{4mM}\Phi(P_t(\omega;\mathbb{A}))
\]
and
\[P_{t}(\omega;\Phi(\mathbb{A}))\le \sum\limits_{i=1}^nw_i\Phi(A_i)\le\frac{(m+M)^2}{4mM} \left(\sum\limits_{i=1}^nw_i\Phi(A_i)^{-1}\right)^{-1}\label{ph}.\qedhere\]
\end{proof}

Note that Theorem \ref{th31} generalize the APH weighted mean inequalities (P4) and by \eqref{Lim}, the Karcher mean satisfies the AKH weighted mean inequalities (See \cite{Lim})
\begin{equation}
\label{akh}
\left(\sum_{i=1}^nw_iA_i^{-1}\right)^{-1}\le G_K(\omega;\mathbb{A})\le \sum_{i=1}^nw_iA_i.
\end{equation}

By \eqref{Lim} and Theorem \ref{th31}, we generalize the AKH weighted mean inequalities \eqref{akh} as follows:
\begin{corollary}
 Under the same conditions as in Theorem \ref{th31}, then
\begin{equation}
\label{agk}
\sum\limits_{i=1}^nw_i\Phi(A_i)\le \frac{(m+M)^2}{4mM} \Phi(G_K(\omega;\mathbb{A}))~~~~~
\end{equation}
and
\[
~~~~~G_K(\omega;\Phi(\mathbb{A}))\le\frac{(m+M)^2}{4mM} \left(\sum\limits_{i=1}^nw_i\Phi(A_i)^{-1}\right)^{-1}\label{ph}.
\]
\end{corollary}

\begin{theorem}\label{th32}
Let $\Phi$ be a positive unital linear map on $B(\mathcal{H})$ and
let $A_1, A_2, \cdots, A_n \in \mathbb{P}^n$ for any positive integer $n\ge 2$ on a Hilbert space $\mathcal{H}$ such that $m\le A_i\le M$ for $i=1,2,\cdots,n$
and some scalars $0<m<M$. Then
\begin{equation}\label{th321}
P_t(\omega;\Phi(\mathbb{A}))\le \frac{(m+M)^2}{4mM}\Phi(P_t(\omega;\mathbb{A})) ~for ~t\in (0,1]~~
\end{equation}
and
\[
~\Phi(P_t(\omega;\mathbb{A}))\le \frac{(m+M)^2}{4mM}P_t(\omega;\Phi(\mathbb{A}))~for ~t\in [-1,0).
\]
\end{theorem}

\begin{proof}
For $t\in (0,1]$, by the weighted APH mean inequalities and \eqref{th311}, it follows that
\[P_t(\omega;\Phi(\mathbb{A}))\le \sum_{i=1}^nw_i\Phi(A_i)\le \frac{(m+M)^2}{4mM} \Phi(P_t(\omega;\mathbb{A})).\]
Let $t \in [-1,0)$ and let $\Phi$ be a strictly positive unital linear map. By the Kantorovich inequality (the reverse of Choi's inequality)
(See \cite[Proposition 2.7.8]{Bhatia}), $\Phi(A^{-1})\le\frac{(m+M)^2}{4mM}\Phi(A)^{-1}$ for all $A > 0$ and \eqref{th321}, we have
\[
k P_{-t}(\omega;\Phi(\mathbb{A})^{-1})=P_{-t}(\omega;k\Phi(\mathbb{A})^{-1})\le P_{-t}(\omega;\Phi(\mathbb{A}^{-1}))\le k\Phi(P_{-t}(\omega;\mathbb{A}^{-1})),
\]
where $k=\frac{(m+M)^2}{4mM}$ and $\Phi(\mathbb{A})^{-1}=(\Phi(A_1)^{-1},\cdots, \Phi(A_n)^{-1})$. Therefore,
\[P_{-t}(\omega;\Phi(\mathbb{A})^{-1})\le \Phi(P_{-t}(\omega;\mathbb{A}^{-1})).\]
This implies that
\[
\Phi(P_{-t}(\omega;\mathbb{A}^{-1})^{-1})\le k\Phi(P_{-t}(\omega;\mathbb{A}^{-1}))^{-1}
\le k P_{-t}(\omega;\Phi(\mathbb{A})^{-1})^{-1}.
\]
By $(\mathrm{P1})$, we obtain the desired inequality
\[
\Phi(P_t(\omega; \mathbb{A}))\le\frac{(m+M)^2}{4mM}P_t(\omega; \Phi(\mathbb{A})).\qedhere
\]
\end{proof}

By \eqref{Lim} and Theorem \ref{th32}, we show a reverse inequality of \eqref{Bhatia} for $n$ positive invertible operators as follows:
\begin{corollary}
 Under the same conditions as in Theorem \ref{th32}, then
\[
G_K(\omega;\Phi(\mathbb{A}))\le \frac{(m+M)^2}{4mM}\Phi(G_K(\omega;\mathbb{A})).
\]
\end{corollary}

Next, we are devoted to obtain several higher power inequalities which are related to
inequalities \eqref{th311} and \eqref{agk}. In order to do that, we need two important lemmas.

\begin{lemma}\cite[Lemma 2.1]{Bhatia2}
\label{lema1}
Let $A, B\ge 0$. Then the following inequality holds:
\begin{equation}
\label{lem12}
\|AB\|\le\frac{1}{4}\|A+B\|^2.
\end{equation}
\end{lemma}

\begin{lemma}\cite[p. 28]{Bhatia}
\label{lema2}
Let $A, B\ge 0$. Then for $1\le r<+\infty$,
\begin{equation}
\label{lem2}
\|A^r+B^r\|\le\|(A+B)^r\|.
\end{equation}
\end{lemma}

It is well known that $\|A\|\le 1$ is equivalent to $A^*A\le I$. This fact plays an important role in the proof of theorems.
\begin{theorem}\label{th33}
Let $\Phi$ be a positive unital linear map on $B(\mathcal{H})$ and
let $A_1, A_2, \cdots, A_n \in \mathbb{P}^n$ for any positive integer $n\ge 2$ on a Hilbert space $\mathcal{H}$ such that $m\le A_i\le M$ for $i=1,2,\cdots,n$
and some scalars $0<m<M$. Then for $t\in(0,1]$ and $p\ge 2$,
\begin{equation}
\label{ap1}
\Phi(\sum\limits_{i=1}^nw_iA_i)^p\le \frac{(m+M)^{2p}}{16m^pM^p} \Phi(P_t(\omega;\mathbb{A}))^p
\end{equation}
and
\begin{equation}
\label{ap2}
~\Phi(\sum\limits_{i=1}^nw_iA_i)^p\le \frac{(m+M)^{2p}}{16m^pM^p} P_t(\omega;\Phi(\mathbb{A}))^p.
\end{equation}
\end{theorem}

\begin{proof}
It is known that \eqref{ap1} is equivalent to
\[
\|\Phi(\sum\limits_{i=1}^nw_iA_i)^\frac{p}{2}\Phi(P_t(\omega;\mathbb{A}))^{-\frac{p}{2}}\|\le \frac{(m+M)^{p}}{4m^\frac{p}{2}M^\frac{p}{2}} .
\]
If $p\ge2$, then 
\begin{align*}
\|M^\frac{p}{2}m^\frac{p}{2}&\Phi(\sum\limits_{i=1}^nw_iA_i)^\frac{p}{2}\Phi(P_t(\omega;\mathbb{A}))^{-\frac{p}{2}}\|\\
&\le \frac{1}{4}\|\Phi(\sum\limits_{i=1}^nw_iA_i)^\frac{p}{2}+M^\frac{p}{2}m^\frac{p}{2}\Phi(P_t(\omega;\mathbb{A}))^{-\frac{p}{2}}\|^2~~(\mathrm{by}~ \eqref{lem12})\\
&\le \frac{1}{4}\|\Phi(\sum\limits_{i=1}^nw_iA_i)+Mm\Phi(P_t(\omega;\mathbb{A}))^{-1}\|^p~~(\mathrm{by}~ \eqref{lem2})\\
&\le \frac{1}{4}\|\Phi(\sum\limits_{i=1}^nw_iA_i)+Mm\Phi(P_t(\omega;\mathbb{A})^{-1})\|^p~~(\text{by the Choi's inequality})\\
&\le \frac{1}{4}\|\Phi(\sum\limits_{i=1}^nw_iA_i)+Mm\Phi(\sum\limits_{i=1}^nw_iA_i^{-1})\|^p~~(\mathrm{by}~ (\rm{P}4))\\
&\le \frac{(M+m)^p}{4}.
\end{align*}
The last inequality above holds as follows:  The condition $0<m\le A_i\le M$ implies that
\[A_i+MmA_i^{-1}\le M+m.\]
Therefore, the following inequality
\begin{equation}
\label{app}
\Phi(\sum\limits_{i=1}^nw_iA_i)+Mm\Phi(\sum\limits_{i=1}^nw_iA_i^{-1})\le M+m
\end{equation}
holds for any positive unital linear map $\Phi$ and $\omega=(w_1,\cdots,w_n)\in \Delta_n$. Thus, \eqref{ap1} holds.

The inequality \eqref{ap2} is equivalent to
\[
\|\Phi(\sum\limits_{i=1}^nw_iA_i)^\frac{p}{2}P_t(\omega;\Phi(\mathbb{A}))^{-\frac{p}{2}}\|\le \frac{(m+M)^{p}}{4m^\frac{p}{2}M^\frac{p}{2}} .
\]
To prove \eqref{ap2}, by computing, we have
\begin{align*}
\|M^\frac{p}{2}m^\frac{p}{2}&\Phi(\sum\limits_{i=1}^nw_iA_i)^\frac{p}{2}P_t(\omega;\Phi(\mathbb{A}))^{-\frac{p}{2}}\|\\
&\le \frac{1}{4}\|\Phi(\sum\limits_{i=1}^nw_iA_i)^\frac{p}{2}+M^\frac{p}{2}m^\frac{p}{2}P_t(\omega;\Phi(\mathbb{A}))^{-\frac{p}{2}}\|^2~~(\mathrm{by}~ \eqref{lem12})\\
&\le \frac{1}{4}\|\Phi(\sum\limits_{i=1}^nw_iA_i)+MmP_t(\omega;\Phi(\mathbb{A}))^{-1}\|^p~~(\mathrm{by}~ \eqref{lem2})\\
&\le \frac{1}{4}\|\Phi(\sum\limits_{i=1}^nw_iA_i)+Mm\sum\limits_{i=1}^nw_i\Phi(A_i)^{-1}\|^p~~(\mathrm{by}~ (\rm{P}4))\\
&\le \frac{1}{4}\|\Phi(\sum\limits_{i=1}^nw_iA_i)+Mm\Phi(\sum\limits_{i=1}^nw_iA_i^{-1})\|^p~~(\text{by the Choi'sinequality})\\
&\le \frac{(M+m)^p}{4}.~~(\mathrm{by}~ \eqref{app})
\end{align*}
Thus, \eqref{ap2} holds.
\end{proof}

By \eqref{Lim}, we obtain the following results about the Karcher mean and the arithmetic mean for any positive unital linear map.
\begin{corollary}\label{cor33}
 Under the same conditions as in Theorem 3.1, then
\[
\Phi(\sum\limits_{i=1}^nw_iA_i)^p\le \frac{(m+M)^{2p}}{16m^pM^p} \Phi(G_K(\omega;\mathbb{A}))^p
\]
and
\[
\Phi(\sum\limits_{i=1}^nw_iA_i)^p\le \frac{(m+M)^{2p}}{16m^pM^p} G_K(\omega;\Phi(\mathbb{A}))^p.
\]
\end{corollary}

By Theorem \ref{th33}, Corollary \ref{cor33} and the L\"{o}ewner-Heinz inequality, we have
\begin{remark}
For $p\in(0,2]$, we have
\begin{equation}
\label{rem331}
\Phi(\sum\limits_{i=1}^nw_iA_i)^p\le \frac{(m+M)^{2p}}{4^pm^pM^p} \Phi(P_t(\omega;\mathbb{A}))^p,
\end{equation}
\begin{equation}
\label{rem332}
\Phi(\sum\limits_{i=1}^nw_iA_i)^p\le \frac{(m+M)^{2p}}{4^pm^pM^p}P_t(\omega;\Phi(\mathbb{A}))^p,
\end{equation}
\begin{equation}
\label{rem351}
\Phi(\sum\limits_{i=1}^nw_iA_i)^p\le \frac{(m+M)^{2p}}{4^pm^pM^p} \Phi(G_K(\omega;\mathbb{A}))^p,
\end{equation}
and
\[
\Phi(\sum\limits_{i=1}^nw_iA_i)^p\le \frac{(m+M)^{2p}}{4^pm^pM^p} G_K(\omega;\Phi(\mathbb{A}))^p.
\]
\end{remark}

\begin{theorem}\label{th34}
Let $\Phi$ be a positive unital linear map on $B(\mathcal{H})$ and
let $A_1, A_2, \cdots, A_n \in \mathbb{P}^n$ for any positive integer $n\ge 2$ on a Hilbert space $\mathcal{H}$ such that $m\le A_i\le M$ for $i=1,2,\cdots,n$
and some scalars $0<m<M$. Then for $t\in(0,1]$, $1<\alpha\le 2$ and $p\ge2\alpha$,
\begin{equation}
\label{th341}
\Phi(\sum\limits_{i=1}^nw_iA_i)^p\le \frac{(k^{\frac{\alpha}{2}}(M^\alpha+m^\alpha))^{\frac{2p}{\alpha}}}{16M^pm^p} \Phi(P_t(\omega;\mathbb{A}))^p
\end{equation}
and
\begin{equation}
\label{th342}
\Phi(\sum\limits_{i=1}^nw_iA_i)^p\le \frac{(k^{\frac{\alpha}{2}}(M^\alpha+m^\alpha))^{\frac{2p}{\alpha}}}{16M^pm^p} P_t(\omega;\Phi(\mathbb{A}))^p.
\end{equation}
\end{theorem}
\begin{proof}
It follows from the inequality \eqref{rem331} that
\begin{equation}
\label{th343}
\Phi(P_t(\omega;\mathbb{A}))^{-\alpha}\le k^\alpha\Phi(\sum\limits_{i=1}^nw_iA_i)^{-\alpha} ~\text{for}~ 1<\alpha\le 2.
\end{equation}
It is known that \eqref{th341} is equivalent to
\[
\|[\Phi(\sum\limits_{i=1}^nw_iA_i)^\frac{p}{2}\Phi(P_t(\omega;\mathbb{A}))^{-\frac{p}{2}}\|\le \frac{(k^{\frac{\alpha}{2}}(M^\alpha+m^\alpha))^{\frac{p}{\alpha}}}{4M^{\frac{p}{2}}m^{\frac{p}{2}}}.
\]
If $p\ge2\alpha$, then 
\begin{align*}
\|M^\frac{p}{2}m^\frac{p}{2}&\Phi(\sum\limits_{i=1}^nw_iA_i)^\frac{p}{2}\Phi(P_t(\omega;\mathbb{A}))^{-\frac{p}{2}}\|\\
&\le \frac{1}{4}\|k^{\frac{p}{4}}\Phi(\sum\limits_{i=1}^nw_iA_i)^\frac{p}{2}+k^{-\frac{p}{4}}M^\frac{p}{2}m^\frac{p}{2}\Phi(P_t(\omega;\mathbb{A}))^{-\frac{p}{2}}\|^2~~(\text{by}~ \eqref{lem12})\\
&\le \frac{1}{4}\|k^{\frac{\alpha}{2}}\Phi(\sum\limits_{i=1}^nw_iA_i)^\alpha+k^{-\frac{\alpha}{2}}M^\alpha m^\alpha\Phi(P_t(\omega;\mathbb{A}))^{-\alpha}\|^\frac{p}{\alpha}~~(\mathrm{by}~ \eqref{lem2})\\
&\le \frac{1}{4}\|k^{\frac{\alpha}{2}}\Phi(\sum\limits_{i=1}^nw_iA_i)^\alpha+k^{\frac{\alpha}{2}}M^\alpha m^\alpha\Phi(\sum\limits_{i=1}^nw_iA_i)^{-\alpha}\|^\frac{p}{\alpha}~~(\mathrm{by}~ \rm \eqref{th343})\\
&\le \frac{k^{\frac{p}{2}}(M^\alpha+m^\alpha)^\frac{p}{\alpha}}{4}.
\end{align*}
The last inequality above holds as follows:  The condition $m\le A_i\le M$ implies that
$m^\alpha\le\Phi(\sum\limits_{i=1}^nw_iA_i)^\alpha\le M^\alpha$.
Therefore, the following inequalty
\[
\Phi(\sum\limits_{i=1}^nw_iA_i)^\alpha+M^\alpha m^\alpha\Phi(\sum\limits_{i=1}^nw_iA_i)^{-\alpha}\le M^\alpha+m^\alpha
\]
holds for any positive unital linear map $\Phi$ and a weight $\omega=(w_1,\cdots,w_n)\in \Delta_n$. Thus, \eqref{th341} holds.

By a similar argument, \eqref{th342} can be derived from the inequality \eqref{rem332}.
\end{proof}

Taking $\alpha=2$ in Theorem 3.4, we have
\begin{corollary}
Under the same conditions as in Theorem \ref{th34}, then for $p\ge 4$,
\begin{equation}
\label{cor341}
\Phi(\sum\limits_{i=1}^nw_iA_i)^p\le \frac{(k(M^2+m^2))^{p}}{16M^pm^p} \Phi(P_t(\omega;\mathbb{A}))^p
\end{equation}
and
\begin{equation}
\label{cor342}
\Phi(\sum\limits_{i=1}^nw_iA_i)^p\le \frac{(k(M^2+m^2))^{p}}{16M^pm^p} P_t(\omega;\Phi(\mathbb{A}))^p.
\end{equation}
\end{corollary}
\begin{remark}
If $\frac{M}{m}\le 2+\sqrt{3}$, we have $k(M^2+m^2)\le (M+m)^2$, so \eqref{cor341} and \eqref{cor342} are sharper than \eqref{th341}
and \eqref{th342} for $p\ge 4$, respectively.
\end{remark}

Note that we can also generalize the inequality \eqref{th341}, \eqref{th342}, \eqref{cor341} and \eqref{cor342} to the Karcher mean by \eqref{Lim} similarly.
\section{Weighted arithmetic and geometric means due to Lawson and Lim}
In 2006, Yamazaki \cite{Yamazaki} obtained a converse of arithmetic-geometric means inequality of $n$-operators via Kantorovich constant by induction on $n$. Soon after, Fujii el al. \cite{Fujii} also proved a stronger reverse inequality of the weighted arithmetic and geometric means due to Lawson and Lim of $n$-operators by the Kantorovich inequality in Theorem \ref{th21}.

In this section, we obtain the higher-power reverse inequalities of the weighted arithmetic and geometric means due to Lawson and Lim of $n$-operators, and several complements of the weighted geometric mean for $n$-variables have been established.
\begin{theorem}\label{th41}
For any positive integer $n\ge 2$, let $A_1, A_2, \cdots, A_n$ be positive
 invertible operators on a Hilbert space $\mathcal{H}$ such that $m\le A_i\le M$ for $i=1,2,\cdots,n$
and some scalars $0<m<M$. Then for $p\ge 2$,
\begin{equation}\label{th411}
A[n,t](A_1,\cdots, A_n)^p\le \frac{(m+M)^{2p}}{16m^pM^p}G[n,t](A_1,\cdots, A_n)^p.
\end{equation}
\end{theorem}

\begin{proof}
Let a map $\Psi: \mathcal{B}(\mathcal{H})\oplus \cdots \oplus\mathcal{B}(\mathcal{H})
\mapsto \mathcal{B}(\mathcal{H})\oplus \cdots \oplus\mathcal{B}(\mathcal{H})$ be defined by
\[\Psi\left(
  \begin{array}{ccc}
    A_1  & &0 \\
     & \ddots& \\
     0& &A_n\\
  \end{array}
\right)=
\left(
  \begin{array}{ccc}
    t[n]_1A_1+\cdots+t[n]_nA_n  & &0 \\
     & \ddots& \\
     0& &t[n]_1A_1+\cdots+t[n]_nA_n\\
  \end{array}
\right).
\]
Then $\Psi$ is a positive linear map such that $\Psi(I)=I$.
The condition $0<m\le A_i\le M$ for $i=1,2,\cdots,n$ implies that
\begin{equation}\label{th412}
m\left(
  \begin{array}{ccc}
    I  & &0 \\
     & \ddots& \\
     0& &I\\
  \end{array}
\right)\le
\Psi\left(
  \begin{array}{ccc}
    A_1  & &0 \\
     & \ddots& \\
     0& &A_n\\
  \end{array}
\right)
\le M\left(
  \begin{array}{ccc}
    I  & &0 \\
     & \ddots& \\
     0& &I\\
  \end{array}
\right).
\end{equation}
By (2.3) in \cite{Lin2}, we have
\[
\Psi\left(
  \begin{array}{ccc}
    A_1  & &0 \\
     & \ddots& \\
     0& &A_n\\
  \end{array}
\right)+
Mm\Psi\left(
  \begin{array}{ccc}
    A_1^{-1}  & &0 \\
     & \ddots& \\
     0& &A_n^{-1}\\
  \end{array}
\right)
\le (M+m)\left(
  \begin{array}{ccc}
    I  & &0 \\
     & \ddots& \\
     0& &I\\
  \end{array}
\right).
\]
Thus,
\begin{equation}
\label{th413}
A[n,t](A_1,\cdots, A_n)+Mm A[n,t](A_1^{-1},\cdots, A_n^{-1})\le m+M.
\end{equation}
On the other hand, by computing, we deduce
\begin{align*}
\|M^\frac{p}{2}m^\frac{p}{2}&A[n,t](A_1,\cdots, A_n)^\frac{p}{2}G[n,t](A_1,\cdots, A_n)^{-\frac{p}{2}}\|\\
&\le \frac{1}{4}\|A[n,t](A_1,\cdots, A_n)^\frac{p}{2}+M^\frac{p}{2}m^\frac{p}{2}G[n,t](A_1,\cdots, A_n)^{-\frac{p}{2}}\|^2~(\text{by}~ \eqref{lem12})\\
&\le \frac{1}{4}\|A[n,t](A_1,\cdots, A_n)+MmG[n,t](A_1,\cdots, A_n)^{-1}\|^p~(\text{by}~ \eqref{lem2})\\
&= \frac{1}{4}\|A[n,t](A_1,\cdots, A_n)+MmG[n,t](A_1^{-1},\cdots, A_n^{-1})\|^p\\
&\le \frac{1}{4}\|A[n,t](A_1,\cdots, A_n)+MmA[n,t](A_1^{-1},\cdots, A_n^{-1})\|^p~(\text{by}~ \eqref{agh})\\
&\le \frac{(M+m)^p}{4}.~(\text{by}~ \eqref{th413})
\end{align*}

The equality above follows from the self-duality of the geometric mean (See \cite{Ando1, Fujii, Yamazaki}).
\end{proof}
Taking $p=2$,  \eqref{th411} implies the following corollary:
\begin{corollary}\label{cor41}
For any positive integer $n\ge 2$, let $A_1, A_2, \cdots, A_n$ be positive
 invertible operators on a Hilbert space $\mathcal{H}$ such that $m\le A_i\le M$ for $i=1,2,\cdots,n$
and some scalars $0<m<M$. Then
\begin{equation}\label{cor411}
A[n,t](A_1,\cdots, A_n)^2\le \frac{(m+M)^{4}}{16m^2M^2}G[n,t](A_1,\cdots, A_n)^2.
\end{equation}
\end{corollary}

Note that if $t=\frac{1}{2}$, the inequality \eqref{cor411} reduces to Lin's result (See \cite[Theorem 3.2]{Lin1}).
\begin{theorem}\label{th42}
For any positive integer $n\ge 2$, let $A_1, A_2, \cdots, A_n$ be positive
 invertible operators on a Hilbert space $\mathcal{H}$ such that $m\le A_i\le M$ for $i=1,2,\cdots,n$
and some scalars $0<m<M$. Then for $1<\alpha\le 2$ and $p\ge2\alpha$,
\begin{equation}\label{th421}
A[n,t](A_1,\cdots, A_n)^p\le \frac{(k^{\frac{\alpha}{2}}(M^\alpha+m^\alpha))^{\frac{2p}{\alpha}}}{16M^pm^p}G[n,t](A_1,\cdots, A_n)^p,
\end{equation}
where $k=\frac{(m+M)^{2}}{4mM}$.
\end{theorem}
\begin{proof}
Let a map $\Psi: \mathcal{B}(\mathcal{H})\oplus \cdots \oplus\mathcal{B}(\mathcal{H})
\mapsto \mathcal{B}(\mathcal{H})\oplus \cdots \oplus\mathcal{B}(\mathcal{H})$ be defined as in the proof of Theorem \ref{th41}.
By \eqref{th412}, we have
\[
m^\alpha\left(
  \begin{array}{ccc}
    I  & &0 \\
     & \ddots& \\
     0& &I\\
  \end{array}
\right)\le
\Psi^\alpha\left(
  \begin{array}{ccc}
    A_1  & &0 \\
     & \ddots& \\
     0& &A_n\\
  \end{array}
\right)
\le M^\alpha\left(
  \begin{array}{ccc}
    I  & &0 \\
     & \ddots& \\
     0& &I\\
  \end{array}
\right),
\]
that is,
\[m^\alpha\le A[n,t](A_1,\cdots, A_n)^\alpha\le M^\alpha.\]
By (2.3) in \cite{Lin2}, we have
\begin{equation}\label{th423}
A[n,t](A_1,\cdots, A_n)^\alpha+M^\alpha m^\alpha A[n,t](A_1,\cdots, A_n)^{-\alpha}\le m^\alpha+M^\alpha.
\end{equation}
On the other hand, by \eqref{cor411},
\begin{equation}\label{th424}
k^{-\alpha}G[n,t](A_1,\cdots, A_n)^{-\alpha}\le A[n,t](A_1,\cdots, A_n)^{-\alpha}.
\end{equation}
By computing, we deduce
\begin{align*}
\|k^{-\frac{p}{2}}m^{\frac{p}{2}}M^{\frac{p}{2}}&A[n,t](A_1,\cdots, A_n)^\frac{p}{2}G[n,t](A_1,\cdots, A_n)^{-\frac{p}{2}}\|\\
&\le \frac{1}{4}\|A[n,t](A_1,\cdots, A_n)^\frac{p}{2}+k^{-\frac{p}{2}}m^{\frac{p}{2}}M^{\frac{p}{2}}G[n,t](A_1,\cdots, A_n)^{-\frac{p}{2}}\|^2~(\text{by}~ \eqref{lem12})\\
&\le \frac{1}{4}\|(A[n,t](A_1,\cdots, A_n)^\alpha+k^{-\alpha}m^{\alpha}M^{\alpha}G[n,t](A_1,\cdots, A_n)^{-\alpha})^{\frac{p}{2\alpha}}\|^2~(\text{by}~ \eqref{lem2})\\
&= \frac{1}{4}\|A[n,t](A_1,\cdots, A_n)^\alpha+k^{-\alpha}m^{\alpha}M^{\alpha}G[n,t](A_1,\cdots, A_n)^{-\alpha}\|^{\frac{p}{\alpha}}\\
&\le \frac{1}{4}\|A[n,t](A_1,\cdots, A_n)^\alpha+m^{\alpha}M^{\alpha}A[n,t](A_1,\cdots, A_n)^{-\alpha}\|^p~(\text{by}~ \eqref{th424})\\
&\le \frac{(M^{\alpha}+m^{\alpha})^p}{4}.~(\text{by}~ \eqref{th423})
\end{align*}
We obtain the desired result.
\end{proof}

Putting $\alpha=2$ in the inequality \eqref{th421}, which implies that
\begin{corollary}
Under the same conditions as in Theorem \ref{th42}, then for $p\ge 4$,
\begin{equation}
\label{cor421}
A[n,t](A_1,\cdots, A_n)^p\le \frac{(k(M^2+m^2))^{p}}{16M^pm^p}G[n,t](A_1,\cdots, A_n)^p.
\end{equation}
\end{corollary}
\begin{remark}

When $\frac{M}{m}\le 2+\sqrt{3}$, it is easy to see that \eqref{cor421} is sharper than \eqref{th411} for $p\ge 4$.
\end{remark}

Next, we show the complements of the weighted geometric mean due to Lawson and Lim by virtue of the following lemma.
\begin{lemma}
For any integer $n\ge 2$, let $A_1,A_2,\cdots,A_n$ be positive
 invertible operators in $\mathbb{P}$
such that $m\le A_i\le M$ for all $i=1,2,\cdots,n$ and some scalars $0<m\le M$.
Then
\begin{equation}
\label{lem411}
A[n,t](A^p_1,\cdots,A^p_n)\le K(m,M,p)A[n,t](A_1,\cdots, A_n)^p~for~all~p\ge1,
\end{equation}
where $K(m,M,p)=\frac{(p-1)^{p-1}}{p^p}\frac{(M^p-m^p)^p}{(M-m)(mM^p-Mm^p)^{p-1}}$ is the generalized Kantorovich constant.
\end{lemma}
\begin{proof}
By Corollary 2.6 in \cite{Micic},
\[\Phi(A^p)\le K(m,M,p)\Phi(A)^p~for~all~p\ge1.\]
Let the map $\Phi: \mathcal{B}(\mathcal{H})\oplus \cdots \oplus\mathcal{B}(\mathcal{H})
\mapsto \mathcal{B}(\mathcal{H})\oplus \cdots \oplus\mathcal{B}(\mathcal{H})$ be defined as $\Psi$ in the proof of Theorem \ref{th41}. Then for $p\ge1$,
 \[A[n,t](A^p_1,\cdots,A^p_n)\le K(m,M,p)A[n,t](A_1,\cdots, A_n)^p.\qedhere\]
\end{proof}

\begin{theorem} \label{th43}
For any integer $n\ge 2$, let $A_1,A_2,\cdots,A_n$ be positive invertible operators in $\mathbb{P}$
such that $m\le A_i\le M $ for all $i=1,2,\cdots,n$ and some scalars $0<m\le M$.
Then
\[
G[n,t](A^p_1,\cdots,A^p_n)\le K(m,M,p)\frac{(m+M)^{2p}}{4^pm^pM^p}G[n,t](A_1,\cdots, A_n)^p~~for~all~1<p\le2,
\]
and
\[
~G[n,t](A^p_1,\cdots,A^p_n)\le K(m,M,p)\frac{(m+M)^{2p}}{16m^pM^p}G[n,t](A_1,\cdots, A_n)^p~for~all~p\ge2.~~~~~~
\]
\end{theorem}
\begin{proof}
By the arithmetic-geometric mean inequality and \eqref{lem411}, it follows that
\begin{equation}
\label{th431}
G[n,t](A^p_1,\cdots,A^p_n)\le A[n,t](A^p_1,\cdots,A^p_n)\le K(m,M,p)A[n,t](A_1,\cdots, A_n)^p~for~p\ge1.
\end{equation}

For $p\in(1,2]$, it follows from \eqref{cor411} and the L\"{o}ewner-Heinz inequality that
\[A[n,t](A_1,\cdots, A_n)^p\le \left(\frac{(m+M)^{2}}{4mM}\right)^pG[n,t](A_1,\cdots, A_n)^p.\]
Combining these inequalities above, we have
\[G[n,t](A^p_1,\cdots,A^p_n)\le K(m,M,p)\frac{(m+M)^{2p}}{4^pm^pM^p}G[n,t](A_1,\cdots, A_n)^p.\]

For $p\in[2,\infty)$, from \eqref{th411} and \eqref{th431}, we obtain
\[G[n,t](A^p_1,\cdots,A^p_n)\le K(m,M,p)\frac{(m+M)^{2p}}{16m^pM^p}G[n,t](A_1,\cdots, A_n)^p.\qedhere\]
\end{proof}

In the following remark, we present the case of $p\ge 1$ about \eqref{th232} for the Ando-Li-Mahthias geometric mean.
\begin{remark}
Let $t=\frac{1}{2}$ in Theorem \ref{th43}. Then
\[
G_{ALM}(A^p_1,\cdots,A^p_n)\le K(m,M,p)\frac{(m+M)^{2p}}{4^pm^pM^p}G_{ALM}(A_1,\cdots, A_n)^p~~for~all~1<p\le2
\]
and
\[
~G_{ALM}(A^p_1,\cdots,A^p_n)\le K(m,M,p)\frac{(m+M)^{2p}}{16m^pM^p}G_{ALM}(A_1,\cdots, A_n)^p~~for~all~p\ge2.~~~~~
\]
\end{remark}

\begin{theorem} \label{th44}
For any integer $n\ge 2$, let $A_1,A_2,\cdots,A_n$ be positive invertible operators  in $\mathbb{P}$
such that $m\le A_i\le M $ for all $i=1,2,\cdots,n$ and some scalars $0<m\le M$.
Then
\[
G[n,t](A^p_1,\cdots,A^p_n)^\frac{1}{p}\le K\left(m^q,M^q,\frac{p}{q}\right)^\frac{1}{p}\left(\frac{(m^q+M^q)^{2}}{4m^qM^q}\right)^\frac{1}{q}G[n,t](A^q_1,\cdots,A^q_n)^\frac{1}{q}~~~
\]
for all $1<\frac{p}{q}\le2$ and $p\ge1$,
and
\[
~~G[n,t](A^p_1,\cdots,A^p_n)^\frac{1}{p}\le 4^{-\frac{2}{p}} K\left(m^q,M^q,\frac{p}{q}\right)^\frac{1}{p}\left(\frac{(m^q+M^q)^{2}}{m^qM^q}\right)^\frac{1}{q}G[n,t](A^q_1,\cdots,A^q_n)^\frac{1}{q}
\]
for all $\frac{p}{q}\ge2$ and $p\ge1$.
\end{theorem}
\begin{proof}
For each $0<q\le p$, it follows from the arithmetic-geometric mean inequality \eqref{agh} and \eqref{lem411} that
\begin{equation}
\label{th441}
\begin{split}
G[n,t](A^p_1,\cdots,A^p_n)&\le A[n,t](A^p_1,\cdots,A^p_n)\\
&= A[n,t]((A^q_1)^\frac{p}{q},\cdots,(A^q_n)^\frac{p}{q})\\
&\le K\left(m^q,M^q,\frac{p}{q}\right)A[n,t](A^q_1,\cdots, A^q_n)^\frac{p}{q}.
\end{split}
\end{equation}
On the other hand, for $1<\frac{p}{q}\le2$, from \eqref{cor411} and $m^qI\le A_i^q\le M^qI$, it follows that
\[
A[n,t](A^q_1,\cdots, A^q_n)^\frac{p}{q}\le \left(\frac{(m^q+M^q)^{2}}{4m^qM^q}\right)^\frac{p}{q}G[n,t](A^q_1,\cdots,A^q_n)^\frac{p}{q}.
\]
Combining the two inequalities above, we obtain
\[
G[n,t](A^p_1,\cdots,A^p_n)\le K\left(m^q,M^q,\frac{p}{q}\right)\left(\frac{(m^q+M^q)^{2}}{4m^qM^q}\right)^\frac{p}{q}G[n,t](A^q_1,\cdots,A^q_n)^\frac{p}{q}.
\]
By the L\"{o}ewner-Heinz inequality and $p\ge1$, it follows that
\[
G[n,t](A^p_1,\cdots,A^p_n)^\frac{1}{p}\le K\left(m^q,M^q,\frac{p}{q}\right)^\frac{1}{p}\left(\frac{(m^q+M^q)^{2}}{4m^qM^q}\right)^\frac{1}{q}G[n,t](A^q_1,\cdots,A^q_n)^\frac{1}{q}.
\]
Similarly, for all $\frac{p}{q}\ge2$, from \eqref{th411} we have
\[
A[n,t](A^q_1,\cdots, A^q_n)^\frac{p}{q}\le 4^{-2}\left(\frac{(m^q+M^q)^{2}}{m^qM^q}\right)^\frac{p}{q}G[n,t](A^q_1,\cdots,A^q_n)^\frac{p}{q}.
\]
Combining with \eqref{th441}, we obtain
\[
G[n,t](A^p_1,\cdots,A^p_n)\le 4^{-2}K\left(m^q,M^q,\frac{p}{q}\right)\left(\frac{(m^q+M^q)^{2}}{m^qM^q}\right)^\frac{p}{q}G[n,t](A^q_1,\cdots,A^q_n)^\frac{p}{q}.
\]
It follows from $p\ge1$ that
\[
G[n,t](A^p_1,\cdots,A^p_n)^\frac{1}{p}\le 4^{-\frac{2}{p}}K\left(m^q,M^q,\frac{p}{q}\right)^\frac{1}{p}\left(\frac{(m^q+M^q)^{2}}{m^qM^q}\right)^\frac{1}{q}G[n,t](A^q_1,\cdots,A^q_n)^\frac{1}{q}.
\]
This completes the proof.
\end{proof}
\begin{remark}
Although $\lim_{q\rightarrow 0}G[n,t](A^q_1,\cdots,A^q_n)^\frac{1}{q}=\diamondsuit(\hat{\omega}; A_1,\cdots, A_n)$ (the chaotic geometric mean),
$\lim_{q\rightarrow 0}\left(\frac{(m^q+M^q)^{2}}{4m^qM^q}\right)^\frac{1}{q}=1$ and $\lim_{q\rightarrow 0}K\left(m^q,M^q,\frac{p}{q}\right)^\frac{1}{p}=S(h^p)^\frac{1}{p}$ (See \cite{Fujii1}), we can not obtain
\[
G[n,t](A^p_1,\cdots,A^p_n)^\frac{1}{p}\le S(h^p)^\frac{1}{p}\diamondsuit(\hat{\omega}; A_1,\cdots, A_n)
\]
for all $p\ge1$ from Theorem \ref{th44}.
\end{remark}

\section{Comparisons between the weighted Karcher mean and the Lawson-Lim geometric mean}
In the final section, we make comparisons between the weighted Karcher mean and the Lawson-Lim geometric mean for higher power.
This is a fascinating work because the order relation can be preserved between higher power operators by the Kantorovich constant.
\begin{lemma}\label{lem51}\cite{Fujii2}
Let $0<m\le A\le M$ and $A\le B$. Then
\[
A^2\le \frac{(M+m)^2}{4Mm}B^2.
\]
\end{lemma}

\begin{lemma}\label{lem52}\cite{Furuta1}
Let $A$ and $B$ be positive invertible operators on a Hilbert space $\mathcal{H}$ satisfying
$B\ge A>0$ and $0<m\le A\le M$. Then
\[
\left(\frac{M}{m}\right)^{p-1}B^p\ge\mathrm{K}(m,M,p)B^p\ge A^p
\]
holds for any $p\ge 1$, where $\mathrm{K}(m,M,p)$ is the generalized Kantorovich constant or the Ky Fan-Furuta constant and
\[\left(\frac{M}{m}\right)^{p-1}\ge \mathrm{K}(m,M,p).\]
\end{lemma}
\begin{theorem}
For any positive integer $n\ge 2$, let $A_1, A_2, \cdots, A_n$ be positive
invertible operators on a Hilbert space $\mathcal{H}$ such that $m\le A_i\le M$ for $i=1,2,\cdots,n$
and some scalars $0<m<M$. Then
\begin{equation}
\label{th511}
\begin{split}
\left(\frac{(m+M)^{2}}{4mM}\right)^{-3}G[n,t](A_1,\cdots, A_n)^2&\le G_K(\hat{\omega}, A_1,\cdots, A_n)^2\\
&\le \left(\frac{(m+M)^{2}}{4mM}\right)^3G[n,t](A_1,\cdots, A_n)^2,
\end{split}
\end{equation}
where $\hat{\omega}=\left(t[n]_1, t[n]_2, \cdots, t[n]_n\right)$. 
\end{theorem}
\begin{proof}
The first inequality in \eqref{th511} follows from Lemma \ref{lem51}, the arithmetic-geometric mean inequality \eqref{agh} and \eqref{rem351} with $p=2$ and an identity map $\Phi$ that
\begin{align*}
G[n,t](A_1,\cdots, A_n)^2&\le \frac{(m+M)^{2}}{4mM}A[n,t](A_1,\cdots, A_n)^2\\
&\le\frac{(m+M)^{2}}{4mM}\left(\frac{(m+M)^{2}}{4mM}\right)^2 G_K(\hat{\omega}, A_1,\cdots, A_n)^2.
\end{align*}
The second inequality in \eqref{th511} follows from Lemma \ref{lem51} and \eqref{cor411} that
\begin{align*}
G_K(\hat{\omega}, A_1,\cdots, A_n)^2&\le \frac{(m+M)^{2}}{4mM}A[n,t](A_1,\cdots, A_n)^2\\
&\le\frac{(m+M)^{2}}{4mM}\left(\frac{(m+M)^{2}}{4mM}\right)^2 G[n,t]A_1,\cdots, A_n)^2.\qedhere
\end{align*}
\end{proof}

\begin{theorem}
For any positive integer $n\ge 2$, let $A_1, A_2, \cdots, A_n$ be positive
invertible operators on a Hilbert space $\mathcal{H}$ such that $m\le A_i\le M$ for $i=1,2,\cdots,n$
and some scalars $0<m<M$. Then
\begin{equation}
\label{th521}
G_K(\hat{\omega},A_1,\cdots, A_n)^p\le \mathrm{K}(m,M,p)\frac{(M+m)^{2p}}{4^pM^pm^p}G[n,t](A_1,\cdots, A_n)^p ~for ~all~ 1\le p\le 2
\end{equation}
and
\begin{equation}
\label{th522}
G_K(\hat{\omega},A_1,\cdots, A_n)^p\le \mathrm{K}(m,M,p)\frac{(M+m)^{2p}}{16M^pm^p}G[n,t](A_1,\cdots, A_n)^p ~for~ all~ p\ge 2.~~~~
\end{equation}
\end{theorem}

\begin{proof}
By the L\"{o}ewner-Heinz inequality and \eqref{cor411}, we have
\begin{equation}
\label{th523}
A[n,t](A_1,\cdots, A_n)^p\le\frac{(M+m)^{2p}}{4^pM^pm^p}G[n,t](A_1,\cdots, A_n)^p ~for ~all~ 1\le p\le 2.
\end{equation}
It follows from Lemma \ref{lem52} and \eqref{th523} that
\begin{align*}
G_K(\hat{\omega},A_1,\cdots, A_n)^p&\le \mathrm{K}(m,M,p)A[n,t](A_1,\cdots, A_n)^p\\
&\le\mathrm{K}(m,M,p)\frac{(M+m)^{2p}}{4^pM^pm^p}G[n,t](A_1,\cdots, A_n)^p.
\end{align*}
The inequality \eqref{th522} follows from Lemma \ref{lem52} and \eqref{th411} that
\begin{align*}
G_K(\hat{\omega},A_1,\cdots, A_n)^p&\le \mathrm{K}(m,M,p)A[n,t](A_1,\cdots, A_n)^p\\
&\le\mathrm{K}(m,M,p)\frac{(M+m)^{2p}}{16M^pm^p}G[n,t](A_1,\cdots, A_n)^p.\qedhere
\end{align*}
\end{proof}

\begin{remark}
When $p=2$, the inequality \eqref{th521} is equivalent to the second one in \eqref{th511}. So
\eqref{th521} generalize the second inequality in \eqref{th511}.
When $p=1, t=\frac{1}{2}$, the inequality \eqref{th521} reduces to the first one in \eqref{th231}.
\end{remark}
Next, we make a comparison of the weighted Karcher mean and the Lawson-Lim geometric mean for unitarily invariant norm by virtue of the following lemma.
\begin{lemma}\label{lem53}\cite[Lemma 3.2]{Fujii1}
Let $A$ and $B$ be positive invertible operators on a Hilbert space $\mathcal{H}$. If $A\le B$,
then there exists a unitary operator $U$ such that $A^p\le UB^pU^*$ for all $p>0$.
\end{lemma}

\begin{theorem}
For any positive integer $n\ge 2$, let $A_1, A_2, \cdots, A_n$ be positive invertible operators on a Hilbert space $\mathcal{H}$  such that $m\le A_i\le M$ for $i=1,2,\cdots,n$
and some scalars $0<m<M$. Then
\[|||G_K(\hat{\omega}, A_1,\cdots, A_n)^p|||\le \mathrm{K}(m,M,p)\frac{(M+m)^{2p}}{4^pM^pm^p}|||G[n,t](A_1,\cdots, A_n)^p|||\]
for $p\ge 1$ and every unitarily invariant norm $|||\cdot|||$.
\end{theorem}

\begin{proof}
By Theorem \ref{th21} and Lemma \ref{lem53}, there exists a unitary operator $U$ such that
\[A[n,t](A_1,\cdots, A_n)^p\le\left(\frac{(m+M)^{2}}{4mM}\right)^pUG[n,t](A_1,\cdots, A_n)^pU^*.\]
Combining Lemma \ref{lem52} with the above inequality, we have
\begin{align*}
G_K(\hat{\omega},A_1,\cdots, A_n)^p&\le \mathrm{K}(m,M,p)A[n,t](A_1,\cdots, A_n)^p\\
&\le\mathrm{K}(m,M,p)\left(\frac{(m+M)^{2}}{4mM}\right)^pUG[n,t](A_1,\cdots, A_n)^pU^*.\qedhere
\end{align*}
\end{proof}


\bibliographystyle{amsplain}

\end{document}